\documentclass[11pt]{amsart}
 \usepackage[dvips]{epsfig}
 \usepackage{comment}
 \usepackage{enumerate}
 \usepackage{amsgen, amstext,amsbsy,amsopn, amsthm, amsfonts,amssymb,amscd,amsmat
 h, euscript, enumerate,url, verbatim, calc,xypic, mathtools}
 \usepackage[foot]{amsaddr}
 \usepackage{mathabx}
\usepackage{mathtools}
\DeclareMathOperator{\Hessian}{Hess}
\DeclareMathOperator{\Tr}{Tr}
 \usepackage{latexsym}
 \usepackage{graphics}
 \usepackage{color}
\usepackage{mathabx}
\newcommand{\proset}{\,\mathrel{\lower 4pt\hbox{$\scriptscriptstyle/$}
\mkern -14mu\subseteq }\,} 

 \newtheorem{theorem}{Theorem}[section]
  \newtheorem{corollary}[theorem]{Corollary}
 \newtheorem{lemma}[theorem]{Lemma}
 \newtheorem{proposition}[theorem]{Proposition}

\newtheorem{notations} {Notations}

\newtheorem{remark}[theorem]{Remark}

 \newtheorem{example}[theorem]{Example}

\numberwithin{equation}{section}

\usepackage{amsmath}
\usepackage{tikz-cd}
\usepackage[ colorlinks=true, linkcolor=blue, citecolor=blue, urlcolor=blue] {hyperref}
\usepackage[backend=biber, style=numeric, sorting=ynt]{biblatex}
\bibliography{mybib}
\makeatother
\title{On Hypersurfaces of Space Forms that are Gradient Yamabe Solitons}
\author{JAHNABI CHAKRABORTI$^1$, ANANDATEERTHA MANGASULI$^2$}
\address{Department of Mathematics\\ Indian Institute of Science Education and Research, Bhopal\\ Madhya Pradesh-462066, India}
\email{$^1$ jahnabi22@iiserb.ac.in, $^2$ anand@iiserb.ac.in}
\keywords{Yamabe gradient solitons; Hypersurfaces of space forms; Scalar curvature; Weak maximum principle}
\date{}
\subjclass[2020]{53C24, 53C25, 53C40, 53C42, 53C44}
\date{}
\begin{document}
\begin{abstract}
    This paper studies nontrivial Yamabe gradient solitons occurring as complete immersed hypersurfaces with constant scalar curvature in space forms $\left(\overline{M}^n(C), \overline{g}\right)$. For solitons with non-vanishing gradients, we establish a structural characterization of the soliton by analyzing a regular level set $\Sigma$ of the soliton function. In particular, imposing a suitable condition on the traceless part of the second fundamental form, $\Phi^{\Sigma}_{\overline{M}}$ of $\Sigma$ as a submanifold of the space form, we prove that the soliton either decomposes as a Riemannian product of $\mathbb{R}$ and a totally umbilical submanifold of the space form, or satisfies a sharp lower bound on $\sup\limits_\Sigma\left|\Phi^{\Sigma}_{\overline{M}}\right|$. Furthermore, we show that the lower bound is attained if, and only if, the soliton is isometric to a Riemannian product of $\mathbb{R}$ and a parallel submanifold of the space form.
\end{abstract}
\maketitle

\section{\textbf{Introduction}}
 Self-similar solutions to the Yamabe flow \cite{HA1}, which evolve by diffeomorphism and scaling, are known as Yamabe solitons. Equivalently, 
A \textit{Yamabe soliton} on a connected, smooth $n$-dimensional manifold $M^n$ is a Riemannian metric $g$ satisfying the equation:
\begin{equation} \label{YS}
    \frac{1}{2} \mathcal{L}_X g = (R - \rho)g,
\end{equation}
where $\mathcal{L}_X$ denotes the Lie derivative along a vector field $X \in \mathcal{X}(M)$, $R$ is the scalar curvature associated with $g$, and $\rho \in 
\mathbb{R}$. A particularly important class arises when the soliton vector field $X$ is the
gradient of a smooth function $f: M \rightarrow \mathbb{R}$. In this setting, $(M^n, g)$ is said to be a \textit{Yamabe gradient soliton}, denoted by 
$(M^n, g, f, \rho)$, and the equation \eqref{YS} reduces to 
\begin{equation}\label{GYS}
    \Hessian f =(R - \rho)g,
\end{equation}
where $\Hessian f$ represents the Hessian of the \textit{soliton function} $f$. Depending on the sign of the scaling parameter $\rho$, a Yamabe gradient soliton is classified 
as \textit{steady} ($\rho = 0$), \textit{shrinking} ($\rho > 0$), or \textit{expanding} ($\rho < 0$). Furthermore, a Yamabe gradient soliton is said to be \textit{trivial} if $f$ is a constant function on $M^n$.

    In 2018, Chen and Deshmukh \cite{chendeshmukh} first examined Yamabe solitons immersed in the Euclidean space, where the soliton vector field $X$ is the tangential component of the Euclidean position vector field. Subsequent to their work, numerous results have been obtained on Yamabe soliton hypersurfaces in space forms under different geometric constraints on the soliton vector field $X$; see, for instance, \cite{sekomaeta}, \cite{fujiimaeta}. Nevertheless, a classification of Yamabe gradient soliton hypersurfaces without such restrictive assumptions on $\nabla f$ remains a primary objective. In 2025, Tokura and Barboza \cite{tokurabarboza} classified complete steady and shrinking Yamabe gradient soliton hypersurfaces with constant mean curvature in space forms. However, the expanding case has not yet been determined. On the other hand, although they fully characterized gradient soliton hypersurfaces with the soliton function having only one critical point, the classification for those without critical points was restricted to steady solitons, and only under the constraints of local conformal flatness and a bounded second fundamental form. Thus, a complete classification for gradient soliton hypersurfaces without critical points remains unresolved.

    In this paper, we present a classification of complete nontrivial Yamabe gradient soliton hypersurfaces isometrically immersed in Space forms with constant scalar curvature and the soliton function having no critical point. While it is established that compact Yamabe gradient solitons necessarily possess constant scalar curvature—and are thus trivial \cite{di} \cite{hsu}—nontrivial gradient solitons with constant scalar curvature do exist, such as the Gaussian soliton. In the context of hypersurfaces of space forms, hyperplanes $\mathbb{R}^{n-1}$ in $\mathbb{R}^n$, horospheres $\mathbb{R}^{n-1}$ in $\mathbb{H}^n$, and cylinders $\mathbb{R}^k\times \mathbb{S}^{n-k-1}$ in $\mathbb{R}^n$ are examples of nontrivial Yamabe gradient solitons with constant scalar curvature. Our main result is as follows:

    \begin{theorem}\label{thm1}
         Let $\left(M^{n-1}, g, f, \rho\right)$ be a complete, oriented, nontrivial Yamabe gradient soliton isometrically immersed into a space form $\left(\overline{M}^n(C), \overline{g}\right), n\ge 6$, with constant scalar curvature. If $f$ has no critical point and a regular level set $\Sigma^{n-2}$ of $f$ has the scalar curvature $R^{\Sigma}>max\left\{(n-2)(n-3)C, 0\right\}$ with
         \begin{equation}\label{hyp}
             \Tr\left(\left(\Phi^{\Sigma}_{\overline{M}}\right)^3\right)\ge -\frac{n-2(k+1)}{\sqrt{(n-2)k(n-k-2)}} \left|\Phi^{\Sigma}_{\overline{M}}\right|^3, \hspace{2mm}\text{for some integer}\hspace{2mm} 1\le k\le n-4,
         \end{equation}
         where $\Phi^{\Sigma}_{\overline{M}}$ defined in \eqref{FE27} is the traceless part of the second fundamental form of the isometric immersion $\Sigma^{n-2} \hookrightarrow \overline{M}^n(C)$.
         Then,
             \begin{enumerate}
                 \item either $\sup\limits_{\Sigma}  \left|\Phi^{\Sigma}_{\overline{M}}\right|^2=0$ and $(M^{n-1}, g)$ is isometric to a Riemannian product of $\mathbb{R}$ and a totally umbilical submanifold of $\overline{M}^n(C)$; 
                 \item or, there exists a unique $\alpha \left(R^{\Sigma}, n-2, k, C\right)>0$ such that
                 \begin{equation*}
                     \sup_{\Sigma}\left|\Phi^{\Sigma}_{\overline{M}}\right|^2\ge \alpha \left(R^{\Sigma}, n-2, k, C\right).
                 \end{equation*}
                  Moreover, the equality $\sup\limits_{\Sigma}\left|\Phi^{\Sigma}_{\overline{M}}\right|^2= \alpha \left(R^{\Sigma}, n-2, k, C\right)$ holds and is attained at some point of $\Sigma^{n-2}$ if, and only if, $\left(M^{n-1}, g\right)$ is isometric to a Riemannian product of $\mathbb{R}$ and a parallel submanifold of $\overline{M}^n(C)$.

             \end{enumerate}
           
    \end{theorem}
\begin{remark}
   The hypothesis \ref{hyp} is automatically satisfied when $k=1$; see Lemma 8 in \cite{alias2012}. Different choices of $k$ give rise to distinct structures of the parallel submanifolds $\Sigma^{n-2}$ of $\overline{M}$. In particular, Example \ref{exam1} provides examples of these distinct structures by taking $m=(k+1)\ge 2$.
\end{remark}  
As a consequence of Theorem \ref{thm1} together with \cite{CSY} and \cite{Takeuchi1981}, we obtain the following corollary.
\begin{corollary}
    If the space form $\left(\overline{M}^n(C), \overline{g}\right)$ is simply connected and the isometric immersion $ \Sigma^{n-2}\hookrightarrow \overline{M}^n(C)$ is full then the equality for $\sup\limits_{\Sigma}\left|\Phi^{\Sigma}_{\overline{M}}\right|^2$ holds and is attained at some point of $\Sigma^{n-2}$ if, and only if, $\left(M^{n-1}, g\right)$ is isometric to either
    \begin{enumerate}
        \item a Riemannian product of $\mathbb{R}$ and irreducible symmetric $R$-spaces;
        \item or a Riemannian product of $\mathbb{R}$ and a hypersurface of a simply connected space form $\overline{M}^{n-1}\left(C+h^2\right)$ with constant scalar curvature, where $h>0$. 
    \end{enumerate}
    
\end{corollary}

\begin{remark}
   In 2018, Alías, Meléndez, and Palmas \cite{AMP} established a classification for hypersurfaces isometrically immersed in space forms with constant scalar curvature. However, one can show that cases $(ii)(b)$ and $(ii)(c)$ in Theorem $1,.2$ of \cite{AMP} do not admit nontrivial gradient soliton structure. This naturally raises the question of whether there exist hypersurfaces isometrically immersed in the spherical and hyperbolic spaces that admit a nontrivial gradient soliton structure with constant scalar curvature and, if so, what the geometric structure of such hypersurfaces is. We demonstrate that by imposing a gradient soliton structure on the metric in Theorem \ref{thm1}, our approach yields an upper bound on the traceless part of the second fundamental form of the regular level sets of the soliton function $f$, rather than on the totally umbilical tensor of the hypersurface $M^{n-1}$ as characterized in their study. A key difference in our approach is the use of the weak maximum principle for the Laplacian, which replaces the Omori–Yau maximum principle for the Laplacian employed by Alías–Meléndez–Palmas.
\end{remark}

\section{\textbf{Examples}}
\begin{notations}\label{not1}
Let $M^m\hookrightarrow \overline{M}^n$ be an isometric immersion.
\begin{enumerate}
    \item $B^M_{\overline{M}}$:= The second fundamental form.
    
    \item $\vec{H}^M_{\overline{M}}$:= The mean curvature vector.

    \item
    $\left|\vec{H}^M_{\overline{M}}\right|$:= The mean curvature function.

     \item $\Phi^M_{\overline{M}}$:= The traceless part of the second fundamental form, defined when $M$ is a hypersurface (i.e., $m=n-1$).

\end{enumerate}
\end{notations}

\begin{example}\label{exam1}
Consider the generalized cylinder
$\mathbb{R}^{m}\times\mathbb{S}^{n-m-1}$ with $m\ge 1$, immersed in the Euclidean space $\mathbb{R}^{n}$. It admits a nontrivial gradient Yamabe soliton structure
with soliton function
\[
f(x,\theta)=\sum_i a_ix_i+b,
\]
where $x=(x_1,\ldots,x_m)\in\mathbb{R}^{m}$,
$a_i>0 \hspace{2mm}\forall i$ , and $b\in\mathbb{R}$.
Let
\[
\mathbb{R}^{m-1}\times\mathbb{S}^{n-m-1}
\hookrightarrow
\mathbb{R}^{m}\times\mathbb{S}^{n-m-1}
\overset{\gamma}{\hookrightarrow}
\mathbb{R}^{n},
\]
where $\gamma$ denotes the unit normal vector field of the generalized
cylinder in $\mathbb{R}^{n}$, and
$\mathbb{R}^{m-1}\times\mathbb{S}^{n-m-1}$ is identified with the slice
$\{x_1=c\}\times\mathbb{R}^{m-1}\times\mathbb{S}^{n-m-1}$.
The generalized cylinder has principal curvatures
\[
\underbrace{0,\ldots,0}_{m},
\underbrace{1,\ldots,1}_{n-m-1},
\]
so that
\[
B^{\mathbb{R}^{m}\times\mathbb{S}^{n-m-1}}_{\mathbb{R}^{n}}
\left(\frac{\partial}{\partial x_i},
\frac{\partial}{\partial x_j}\right)=0,
\qquad
1\le i,j\le m,
\]
and
\[
\left|
\vec{H}^{\mathbb{R}^{m}\times\mathbb{S}^{n-m-1}}_{\mathbb{R}^{n}}
\right|
=
\frac{n-m-1}{n-1}.
\]
Since
$\mathbb{R}^{m-1}\times\mathbb{S}^{n-m-1}$
is totally geodesic in
$\mathbb{R}^{m}\times\mathbb{S}^{n-m-1}$,
the composition formula for second
fundamental forms implies that
\[
\left|
B^{\mathbb{R}^{m-1}\times\mathbb{S}^{n-m-1}}_{\mathbb{R}^{n}}
\right|^2
=
\left|
B^{\mathbb{R}^{m}\times\mathbb{S}^{n-m-1}}_{\mathbb{R}^{n}} \right|^2
=n-m-1.
\]
On the other hand,
\[
\left|
\vec{H}^{\mathbb{R}^{m-1}\times\mathbb{S}^{n-m-1}}_{\mathbb{R}^{n}}
\right|
=
\frac{n-m-1}{n-2}
\neq
\frac{n-m-1}{n-1}
=
\left|
\vec{H}^{\mathbb{R}^{m}\times\mathbb{S}^{n-m-1}}_{\mathbb{R}^{n}}
\right|.
\]
Consequently,
\[
\left|
\Phi^{\mathbb{R}^{m-1}\times\mathbb{S}^{n-m-1}}_{\mathbb{R}^{n}}
\right|^2
=
\frac{(m-1)(n-m-1)}{n-2},
\]
whereas
\[
\left|
\Phi^{\mathbb{R}^{m}\times\mathbb{S}^{n-m-1}}_{\mathbb{R}^{n}}
\right|^2
=
\frac{m(n-m-1)}{n-1},
\]
where $\Phi^{\mathbb{R}^{m-1}\times\mathbb{S}^{n-m-1}}_{\mathbb{R}^{n}}$ is defined by \eqref{FE27}.
Hence
\[
\left|
\Phi^{\mathbb{R}^{m-1}\times\mathbb{S}^{n-m-1}}_{\mathbb{R}^{n}}
\right|^2
\neq
\left|
\Phi^{\mathbb{R}^{m}\times\mathbb{S}^{n-m-1}}_{\mathbb{R}^{n}}
\right|^2.
\]
\end{example}

\begin{example}
Consider the horosphere $\mathbb{R}^{n-1}$ immersed in the hyperbolic space
$\mathbb{H}^n$. It admits a nontrivial gradient Yamabe soliton structure with
soliton function
\[
f(x)=\sum_{i=1}^{n-1}a_i x_i+b,
\]
where $x_i\in\mathbb{R}$, $a_i>0$ for every
$i=1,\ldots,n-1$, and $b\in\mathbb{R}$.
Let
\[
\mathbb{R}^{n-2}\hookrightarrow
\mathbb{R}^{n-1}
\overset{\gamma}{\hookrightarrow}
\mathbb{H}^n,
\]
where $\gamma$ denotes the unit normal vector field of the horosphere
$\mathbb{R}^{n-1}$ in $\mathbb{H}^n$, and
$\mathbb{R}^{n-2}$ is a hyperplane of $\mathbb{R}^{n-1}$.
Since the horosphere is totally umbilical,
\[
B^{\mathbb{R}^{n-1}}_{\mathbb{H}^n}(X,Y)
=g(X,Y)\gamma.
\]
In particular,
\[
B^{\mathbb{R}^{n-1}}_{\mathbb{H}^n}
\left(\frac{\partial}{\partial r},
\frac{\partial}{\partial r}\right)
=\gamma,
\]
and
\[
|\vec{H}^{\mathbb{R}^{n-1}}_{\mathbb{H}^n}|=1.
\]
Since $\mathbb{R}^{n-2}$ is totally geodesic in
$\mathbb{R}^{n-1}$, the composition formula yields
\[
\left|B^{\mathbb{R}^{n-2}}_{\mathbb{H}^n}\right|^2
=n-2
\neq
n-1
=
\left|B^{\mathbb{R}^{n-1}}_{\mathbb{H}^n}\right|^2,
\]
whereas
\[
|\vec{H}^{\mathbb{R}^{n-2}}_{\mathbb{H}^n}|
=
|\vec{H}^{\mathbb{R}^{n-1}}_{\mathbb{H}^n}|
=1.
\]
Finally, both immersions are totally umbilical. Hence
\[
\Phi^{\mathbb{R}^{n-2}}_{\mathbb{H}^n}\equiv0 \equiv \Phi^{\mathbb{R}^{n-1}}_{\mathbb{H}^n},
\]
where $\Phi^{\mathbb{R}^{n-2}}_{\mathbb{H}^n}$ is defined by \eqref{FE27}.
\end{example}

\begin{example}\label{exam3}
    Consider the complete hypersurface
    $\mathbb{R}\times\mathbb{S}^{n-2}(s)$ immersed in the unit sphere
    $\mathbb{S}^{n}$ through the immersion
    \[
    \phi(t,\theta)
    =
    \left(
    r\cos\frac{t}{r},
    r\sin\frac{t}{r},
    s\theta
    \right),
    \]
    where $\theta\in\mathbb{S}^{n-2}$, $r,s>0$, and
    $r^{2}+s^{2}=1$.  Equivalently, the immersion is the composition
    \[
    \mathbb{R}\times\mathbb{S}^{n-2}(s)
    \xrightarrow{\;\,p\times\mathrm{id}\,\;}
    \mathbb{S}^{1}(r)\times\mathbb{S}^{n-2}(s)
    \hookrightarrow
    \mathbb{S}^{n},
    \]
    where $p:\mathbb{R}\rightarrow\mathbb{S}^{1}(r)$ is the universal covering map.
    It admits a gradient Yamabe soliton structure with soliton function
    \[
    f(t,\theta)=at+b,
    \]
    where $t\in\mathbb{R}$, $a>0$, and $b\in\mathbb{R}$.
Let
    \[
    \mathbb{S}^{n-2}(s)
    \hookrightarrow
    \mathbb{R}\times\mathbb{S}^{n-2}(s)
    \overset{\gamma}{\hookrightarrow}
    \mathbb{S}^{n},
    \]
    where $\gamma$ denotes the unit normal vector field of
    $\mathbb{R}\times\mathbb{S}^{n-2}(s)$ in $\mathbb{S}^{n}$, and
    $\mathbb{S}^{n-2}$ is identified with the slice
    $\{t_{0}\}\times\mathbb{S}^{n-2}$ for some fixed
    $t_{0}\in\mathbb{R}$.
    The hypersurface has principal curvatures
    \[
    -\frac{s}{r},
    \underbrace{\frac{r}{s},\ldots,\frac{r}{s}}_{n-2},
    \]
    so that
    \[
    B^{\mathbb{R}\times\mathbb{S}^{n-2}}_{\mathbb{S}^{n}}
    \left(
    \frac{\partial}{\partial t},
    \frac{\partial}{\partial t}
    \right)
    =
    -\frac{s}{r}\,\gamma,
    \]
    and
    \[
    \left|
    \vec{H}^{\mathbb{R}\times\mathbb{S}^{n-2}}_{\mathbb{S}^{n}}
    \right|
    =
    \left|
    \frac{-\frac{s}{r}+(n-2)\frac{r}{s}}{n-1}
    \right|.
    \]
    Since the slice $\mathbb{S}^{n-2}(s)$ is totally geodesic in
    $\mathbb{R}\times\mathbb{S}^{n-2}(s)$, the composition formula for second
    fundamental forms yields
    \[
    \left|
    B^{\mathbb{S}^{n-2}}_{\mathbb{S}^{n}}
    \right|^{2}
    =
    (n-2)\left(\frac{r}{s}\right)^{2}.
    \]
    Moreover,
    \[
    \left|
    \vec{H}^{\mathbb{S}^{n-2}}_{\mathbb{S}^{n}}
    \right|
    =
    \frac{r}{s},
    \]
    whereas
    \[
    \left|
    \vec{H}^{\mathbb{R}\times\mathbb{S}^{n-2}}_{\mathbb{S}^{n}}
    \right|
    =
    \left|
    \frac{-\frac{s}{r}+(n-2)\frac{r}{s}}{n-1}
    \right|.
    \]
    Hence, in general,
    \[
    \left|
    \vec{H}^{\mathbb{S}^{n-2}}_{\mathbb{S}^{n}}
    \right|
    \neq
    \left|
    \vec{H}^{\mathbb{R}\times\mathbb{S}^{n-2}}_{\mathbb{S}^{n}}
    \right|.
    \]
    Consequently,
    \[
    \left|
    \Phi^{\mathbb{S}^{n-2}}_{\mathbb{S}^{n}}
    \right|^{2}
    =0,
    \]
    while
    \[
    \left|
    \Phi^{\mathbb{R}\times\mathbb{S}^{n-2}}_{\mathbb{S}^{n}}
    \right|^{2}
    =
    \frac{n-2}{n-1}
    \left(
    \frac{r}{s}
    +
    \frac{s}{r}
    \right)^{2},
    \]
where $\Phi^{\mathbb{S}^{n-2}}_{\mathbb{S}^{n}}$ is defined by \eqref{FE27}. Therefore,
    \[
    \left|
    \Phi^{\mathbb{S}^{n-2}}_{\mathbb{S}^{n}}
    \right|^{2}
    \neq
    \left|
    \Phi^{\mathbb{R}\times\mathbb{S}^{n-2}}_{\mathbb{S}^{n}}
    \right|^{2}.
    \]
\end{example}



\begin{example}
    Consider the complete hypersurface
    $\mathbb{H}^{n-2}\left(-\sqrt{1+r^2}\right)\times \mathbb{R}$
    immersed in the hyperbolic space $\mathbb{H}^{n}$ with $r>0$.
   The immersion is the composition
\[
\mathbb{H}^{n-2}\left(-\sqrt{1+r^2}\right)\times\mathbb{R}
\xrightarrow{\;\mathrm{id}\times p\;}
\mathbb{H}^{n-2}\left(-\sqrt{1+r^2}\right)\times\mathbb{S}^{1}(r)
\hookrightarrow
\mathbb{H}^{n},
\]
where $p:\mathbb{R}\to\mathbb{S}^{1}(r)$ is the universal covering map.
    It admits a gradient Yamabe soliton structure with soliton function
    \[
    f(\theta,t)=at+b,
    \]
    where $t\in\mathbb{R}$, $\theta\in
    \mathbb{H}^{n-2}\left(-\sqrt{1+r^2}\right)$,
    $a>0$, and $b\in\mathbb{R}$.\\
Let
    \[
    \mathbb{H}^{n-2}\left(-\sqrt{1+r^2}\right)
    \hookrightarrow
    \mathbb{H}^{n-2}\left(-\sqrt{1+r^2}\right)\times\mathbb{R}
    \overset{\gamma}{\hookrightarrow}
    \mathbb{H}^{n},
    \]
    where $\gamma$ denotes the unit normal vector field of
    $\mathbb{H}^{n-2}\left(-\sqrt{1+r^2}\right)\times\mathbb{R}$
    in $\mathbb{H}^{n}$, and
    $\mathbb{H}^{n-2}$ is identified with the slice
    $\mathbb{H}^{n-2}\left(-\sqrt{1+r^2}\right)\times\{t_0\}$.
The hypersurface has principal curvatures
    \[
    \frac{\sqrt{1+r^2}}{r},
    \underbrace{\frac{r}{\sqrt{1+r^2}},
    \ldots,
    \frac{r}{\sqrt{1+r^2}}}_{n-2},
    \]
    so that
    \[
    B^{\mathbb{H}^{n-2}\times\mathbb{R}}_{\mathbb{H}^{n}}
    \left(
    \frac{\partial}{\partial t},
    \frac{\partial}{\partial t}
    \right)
    =
    \frac{\sqrt{1+r^2}}{r}\,\gamma,
    \]
    and
\[
\left|
\vec H^{\mathbb H^{n-2}\times\mathbb R}_{\mathbb H^n}
\right|
=
\frac{(n-1)r^2+1}
{(n-1)r\sqrt{1+r^2}}.
\]
Since the slice
    $\mathbb{H}^{n-2}\left(-\sqrt{1+r^2}\right)$
    is totally geodesic in
    $\mathbb{H}^{n-2}\left(-\sqrt{1+r^2}\right)\times\mathbb{R}$, the composition formula for second fundamental forms yields
    \[
    \left|
    B^{\mathbb{H}^{n-2}}_{\mathbb{H}^{n}}
    \right|^{2}
    =
    (n-2)
    \frac{r^{2}}{1+r^{2}}.
    \]
Moreover,
    \[
    \left|
    \vec H^{\mathbb{H}^{n-2}}_{\mathbb{H}^{n}}
    \right|
    =
    \frac{r}{\sqrt{1+r^{2}}},
    \]
    whereas
    \[
    \left|
    \vec H^{\mathbb{H}^{n-2}\times\mathbb{R}}_{\mathbb{H}^{n}}
    \right|
    =
    \frac{
    \frac{\sqrt{1+r^2}}{r}
    +(n-2)\frac{r}{\sqrt{1+r^2}}
    }{n-1}.
    \]
Hence, in general,
    \[
    \left|
    \vec H^{\mathbb{H}^{n-2}}_{\mathbb{H}^{n}}
    \right|
    \neq
    \left|
    \vec H^{\mathbb{H}^{n-2}\times\mathbb{R}}_{\mathbb{H}^{n}}
    \right|.
    \]
Consequently,
    \[
    \left|
    \Phi^{\mathbb{H}^{n-2}}_{\mathbb{H}^{n}}
    \right|^{2}
    =
    0,
    \]
    while
    \[
    \left|
    \Phi^{\mathbb{H}^{n-2}\times\mathbb{R}}_{\mathbb{H}^{n}}
    \right|^{2}
    =
    \frac{n-2}
    {(n-1)r^{2}(1+r^{2})},
    \]
    where $\Phi^{\mathbb{H}^{n-2}}_{\mathbb{H}^{n}}$ is defined by \eqref{FE27}.
    Therefore,
    \[
    \left|
    \Phi^{\mathbb{H}^{n-2}}_{\mathbb{H}^{n}}
    \right|^{2}
    \neq
    \left|
    \Phi^{\mathbb{H}^{n-2}\times\mathbb{R}}_{\mathbb{H}^{n}}
    \right|^{2}.
    \]
\end{example}


\section{\textbf{Preliminaries} }
\subsection{Level sets of the soliton function as immersed submanifolds of Space forms}

Let $\left(\overline{M}^n(C), \overline{g}\right)$ denote a space form of constant sectional curvature $C$. We consider a connected, complete, nontrivial Yamabe gradient soliton hypersurface  $(M^{n-1}, g, f, \rho)$ isometrically immersed into $\left(\overline{M}^n(C), \overline{g}\right)$. We assume that the hypersurface $(M^{n-1}, g)$ possesses constant scalar curvature and that the soliton function $f$ admits no critical points on $M^{n-1}$. Furthermore, $M^{n-1}$ is assumed to be oriented by a globally defined unit normal vector field $\gamma$. For any regular level set $\Sigma^{n-2}$ of $f$, we view $\Sigma^{n-2}$ as an immersed submanifold of codimension $2$ in the ambient space $\overline{M}^n(C)$. Following the notation introduced in Notations \ref{not1}, we denote the second fundamental form and the mean curvature vector 
of the isometric immersion $\Sigma^{n-2} \hookrightarrow \overline{M}^n(C)$ by $B^{\Sigma}_{\overline{M}}$ and $\vec{H}^{\Sigma}_{\overline{M}}$
respectively.

\begin{lemma}\label{lemma1}
    The normal bundle of the isometric immersion $\Sigma^{n-2} \hookrightarrow \overline{M}^n(C)$  is spanned by two globally defined linearly independent unit normal vector fields $\left\{ \frac{\nabla f}{|\nabla f|}, \gamma\right\}$ such that the components of $B^{\Sigma}_{\overline{M}}$ and $\vec{H}^{\Sigma}_{\overline{M}}$ along $\frac{\nabla f}{|\nabla f|}$ vanishes.
    
\end{lemma}
    
\begin{proof}
Following the local warped product construction \cite{CSY} of a complete, nontrivial Yamabe gradient soliton $(M^{n-1}, g)$, we observe that the soliton function $f$ depends on the $1$-dimensional Euclidean space $(\mathbb{R}, dr^2)$. So, from the definition of the Yamabe gradient soliton \eqref{GYS}, we can write,
\begin{equation}\label{FE1}
    f^{''}(r)=(R-\rho),
\end{equation}
where $R$ is the scalar curvature of $(M^{n-1}, g)$.\\
Since $R$ is constant and $f$ has no critical point on $M^{n-1}$, using the elliptic equation of the scalar curvature of a Yamabe gradient soliton in \cite[Page 363]{daskalopoulos} and \eqref{FE1}, we get $R=\rho$.\\
Therefore, the normal vector to the regular level set $\Sigma^{n-2}$ of $f$ is,
\begin{equation}\label{FE7}
   \nabla f= c \frac{\partial}{\partial r}, \hspace{2mm}\text{where}\hspace{2mm} c\in \mathbb{R}\setminus \{0\}. 
\end{equation}
Also, 
\begin{equation}\label{FE7A}
    B^{\Sigma}_{\overline{M}}(X, Y)= B^{\Sigma}_M(X, Y)+ B^M_{\overline{M}}(X, Y), \hspace{2mm}\forall X, Y\in \mathcal{X}(\Sigma).
\end{equation}
Hence, in view of \eqref{FE7}, \eqref{FE7A}, and the equation $(2.6)$ of \cite{CSY} we get that  $B^{\Sigma}_{\overline{M}}$ satisfies
\begin{equation*}
    \overline{g}\left( B^{\Sigma}_{\overline{M}}(E_i, E_j), \frac{\nabla f}{|\nabla f|} \right)=0,
\end{equation*}
where $\{E_1, ..., E_{n-2} \}$ is a local orthonormal frame of $T{\Sigma}^{n-2}$.
\end{proof}
In this setting, we choose a local orthonormal frame $\left\{E_1, E_2, ..., E_{n-2}, \frac{\nabla f}{|\nabla f|}, \gamma\right\}$ adapted to the Rieammian metric of $\overline{M}^n(C)$ along with its dual coframe $\{ \omega_1, \omega_2,..., \omega_n\}$ such that, upon restriction to the submanifold $\Sigma^{n-2}$, $\{E_1, E_2, ..., E_{n-2}\}$ are tangent to $\Sigma^{n-2}$. Consequently, by virtue of Lemma \ref{lemma1}, the second fundamental form and the mean curvature vector 
can be expressed as:
    
\begin{equation}
    B^{\Sigma}_{\overline{M}}=\sum_{i, j} B_{ij} \gamma \hspace{2mm} \text{and} \hspace{2mm} \vec{H}^{\Sigma}_{\overline{M}}=\frac{1}{n-2}\sum_iB_{ii} \gamma, 
\end{equation}
 where $B_{ij}=\overline{g}\left(B^{\Sigma}_{\overline{M}}(E_i, E_j), \gamma\right)$.\\
The scalar curvature of $(\Sigma^{n-2}, g_{\Sigma})$ is given by,
\begin{equation}\label{FE4}
    R^{\Sigma}=(n-2)(n-3)C+(n-2)^2 \left|\vec{H}^{\Sigma}_{\overline{M}}\right|^2-\left|B^{\Sigma}_{\overline{M}}\right|^2.
\end{equation}
The traceless part of the second fundamental form is the symmetric tensor
\begin{equation}\label{FE27}
    \Phi^{\Sigma}_{\overline{M}}=\sum\limits_{i, j} \left(\Phi^{\Sigma}_{\overline{M}}\right)_{ij} \omega_i\otimes \omega_j\otimes \gamma, 
\end{equation}
where $\left(\Phi^{\Sigma}_{\overline{M}}\right)_{ij}=B_{ij}-\overline{g}\left(\vec{H}^{\Sigma}_{\overline{M}}, \gamma \right) \delta_{ij}$.
\subsection{Cheng-Yau type operator on \textbf{$\Sigma^{n-2}$}}
 In analogy with the Chang-Yau operator \cite{CY}, we define a linear operator $L:C^2(\Sigma)\rightarrow C^2(\Sigma)$ by 
\begin{equation}\label{FE6}
    L(u)=\sum_{i, j} \left\{\big(n-2)\overline{g}\left(\vec{H}^{\Sigma}_{\overline{M}}, \gamma \right)\delta_{ij}-B_{ij}\right\} u_{ij},
\end{equation}
 where $u_{ij}=\Hessian u(E_i, E_j)$.\\
In what follows, we adapt Lemma $6.1$ from \cite{aliasbook} and Lemma $5$ from \cite{alias2012} to the case where $\Sigma^{n-2}$ is a submanifold of codimension 2 in space forms.
\begin{lemma}\label{lemma2}
    The second fundamental form $B^{\Sigma}_{\overline{M}}$ of $\Sigma^{n-2}$ satisfies the following:
   \begin{equation*}
        \frac{1}{2} \Delta \left|B^{\Sigma}_{\overline{M}}\right|^2=\left|\nabla B^{\Sigma}_{\overline{M}}\right|^2+(n-2)\sum_{i, j} B_{ij} \left\{\overline{g}\left(\vec{H}^{\Sigma}_{\overline{M}}, \gamma \right)\right\}_{ij}+c(n-2)\left|\Phi^{\Sigma}_{\overline{M}}\right|^2+(n-2) \overline{g}\left(\vec{H}^{\Sigma}_{\overline{M}}, \gamma \right)\sum_{i, j, k} B_{ij}B_{jk}B_{ki}-\left|B^{\Sigma}_{\overline{M}}\right|^4,
    \end{equation*}
    where $\left\{\overline{g}\left(\vec{H}^{\Sigma}_{\overline{M}}, \gamma \right)\right\}_{ij}= \Hessian\overline{g}\left(\vec{H}^{\Sigma}_{\overline{M}}, \gamma \right) (E_i, E_j)$.
\end{lemma}
\begin{proof}
    Considering $\frac{1}{2} \Delta \left|B^{\Sigma}_{\overline{M}}\right|^2= \sum\limits_{i, j} B_{ij} \Delta B_{ij}+ \sum\limits_{i, j, k} (B_{ijk})^2$ and using the Codazzi equation together with the Ricci formula \cite{cheng} we get,
    \begin{align}\label{FE3}
        \frac{1}{2} \Delta \left|B^{\Sigma}_{\overline{M}}\right|^2 &= \sum_{i, j, k} B_{ij} B_{kijk}+ \left|\nabla B^{\Sigma}_{\overline{M}}\right|^2 \nonumber\\
                                          &= \left|\nabla B^{\Sigma}_{\overline{M}}\right|^2+ \sum_{i, j, k} B_{ij} \left\{B_{kikj}+\sum_m B_{mi} R^{\Sigma}_{mkjk}+\sum_m B_{km} R^{\Sigma}_{mijk}\right\} \nonumber\\
                                          &= \left|\nabla B^{\Sigma}_{\overline{M}}\right|^2+ \sum_{i, j, k} B_{ij} B_{kikj}+\sum_{i, j, m} B _{ij} B_{mi} R^{\Sigma}_{mj}+\sum_{i, j, k, m} B_{ij} B_{km} R^{\Sigma}_{mijk},
    \end{align}
    where $ R^{\Sigma}_{ijkm} $ and $R^{\Sigma}_{mj}$ are the components of the Riemannian curvature tensor and the Ricci tensor of $(\Sigma^{n-2}, g_{\Sigma})$, respectively.\\
    We have,
    \begin{equation}\label{FE2}
        R^{\Sigma}_{mijk}= C (\delta_{mj} \delta_{ik}- \delta_{mk} \delta_{ij})+(B_{mj} B_{ik}- B_{mk} B_{ij}),
    \end{equation}
    and
    \begin{equation}\label{FE1a}
        R^{\Sigma}_{mj}= (n-3) C \delta_{jm}+ (n-2) \overline{g}\left(\vec{H}^{\Sigma}_{\overline{M}}, \gamma \right) B_{jm}- \sum_l B_{lj} B_{ml}.
    \end{equation}
    Substituting \eqref{FE1a} and \eqref{FE2} into \eqref{FE3} and then applying the Codazzi equation, we obtain,
    \begin{equation*}
        \frac{1}{2} \Delta \left|B^{\Sigma}_{\overline{M}}\right|^2= \left|\nabla B^{\Sigma}_{\overline{M}}\right|^2+ \sum_{i, j, k} B_{ij} B_{kkij}+ C(n-2) \left|\Phi^{\Sigma}_{\overline{M}}\right|^2-\left|B^{\Sigma}_{\overline{M}}\right|^4+(n-2) \overline{g}\left(\vec{H}^{\Sigma}_{\overline{M}}, \gamma \right)  \sum_{i, j, k} B_{ij} B_{jk} B_{ki}.
    \end{equation*}
\end{proof}

\begin{lemma}\label{lemma3}
    If the scalar curvature of $\Sigma$, $R^{\Sigma}>(n-2)(n-3) C$ then the linear operator $L$ is elliptic.
\end{lemma}
\begin{proof}
    Since $R^{\Sigma}>(n-2)(n-3) C$ it follows from \eqref{FE4} that 
    \begin{equation}\label{FE21}
        \left|\vec{H}^{\Sigma}_{\overline{M}}\right|^2>0
    \end{equation}and
    \begin{equation} \label{FE5}
        (n-2)^2\left|\vec{H}^{\Sigma}_{\overline{M}}\right|^2\ge {k_i}^2,
    \end{equation}
    where $k_i$'s are the eigenvalues of the self-adjoint linear map associated with the bilinear form $B_{\gamma}(X, Y)= \overline{g}\left(B^{\Sigma}_{\overline{M}}(X, Y), \gamma\right)$.\\
    Now, $\left(\vec{H}^{\Sigma}_{\overline{M}}\right)^2=\left\{\overline{g}\left(\vec{H}^{\Sigma}_{\overline{M}}, \gamma \right)\right\}^2$. Therefore, \eqref{FE21} implies $\overline{g}\left(\vec{H}^{\Sigma}_{\overline{M}}, \gamma \right)>0$ or $\overline{g}\left(\vec{H}^{\Sigma}_{\overline{M}}, \gamma \right)<0$. Given $(M^n, g)$ is an oriented hypersurface of $\overline{M}^n(C)$, we can choose the orientation such that $\overline{g}\left(\vec{H}^{\Sigma}_{\overline{M}}, \gamma \right)>0$.\\
    Hence, \eqref{FE5} implies that $0<(n-2)\overline{g}\left(\vec{H}^{\Sigma}_{\overline{M}}, \gamma \right)-k_i<2(n-2)\left|\vec{H}^{\Sigma}_{\overline{M}}\right|$, where $(n-2)\overline{g}\left(\vec{H}^{\Sigma}_{\overline{M}}, \gamma \right)-k_i$'s are the eigenvalues of the operator $L$ defined in \eqref{FE6}.
\end{proof}
 \subsection{Weak maximum principle for the operator $L$}
To prove Theorem \ref{thm1}, we need a weak maximum principle for the linear operator $L$ as follows.
\begin{proposition}\label{prop1}
      If the scalar curvature of $\Sigma^{n-2}$, $R^{\Sigma}>(n-2)(n-3) C$ and $\sup\limits_{\Sigma}\left|\Phi^{\Sigma}_{\overline{M}}\right|^2<+\infty$ then there exists a sequence of points $\{p_k\}_{k}\subset \Sigma^{n-2}$ such that for every $k\in \mathbb{N}$ the following holds:
      \begin{equation}
          \left|\Phi^{\Sigma}_{\overline{M}}\right|^2(p_k)>\sup_{\Sigma} \left|\Phi^{\Sigma}_{\overline{M}}\right|^2-\frac{1}{k} \hspace{2mm} \text{and} \hspace{2mm} L\left(\left|\Phi^{\Sigma}_{\overline{M}}\right|^2\right)(p_k)<\frac{c}{k},
      \end{equation}
      for a fixed constant $c>0$.
\end{proposition}

\begin{proof}
    We choose a local orthonormal frame $\{E_1, E_2, ..., E_{n-2}\}$ on $\Sigma^{n-2}$ such that \eqref{FE6} can be written as
    \begin{equation}\label{FE11}
        L(\left|\Phi^{\Sigma}_{\overline{M}}\right|^2)= \sum_i \left\{(n-2)\overline{g}\left(\vec{H}^{\Sigma}_{\overline{M}}, \gamma \right)-k_i\right\} \left(\left|\Phi^{\Sigma}_{\overline{M}}\right|^2\right)_{ii},
    \end{equation}
    where $k_i$'s are the eigenvalues of the self-adjoint linear map associated with the bilinear form $B_{\gamma}(X, Y)= \overline{g}\left(B^{\Sigma}_{\overline{M}}(X, Y), \gamma\right)$.\\
 Claim 1: The term in \eqref{FE11}, $(n-2)\overline{g}\left(\vec{H}^{\Sigma}_{\overline{M}}, \gamma \right)-k_i>0$ and bounded.\\
  Given that $R^{\Sigma}> (n-2)(n-3)C$, \eqref{FE4} yields
 \begin{align}\label{FE24}
     & |k_i|<(n-2)\overline{g}\left(\vec{H}^{\Sigma}_{\overline{M}}, \gamma \right) \hspace{2mm}\forall i,\\
     \implies& 0<(n-2)\overline{g}\left(\vec{H}^{\Sigma}_{\overline{M}}, \gamma \right)-k_i< 2(n-2)\overline{g}\left(\vec{H}^{\Sigma}_{\overline{M}}, \gamma \right) \nonumber \hspace{2mm} \forall i.
 \end{align}
 Since $\left|\Phi^{\Sigma}_{\overline{M}}\right|^2<A$, for some $A\in \mathbb{R}$ therefore from \eqref{FE4} we have,
 \begin{align}\label{FE25}
     &(n-2)(n-3) \left\{\overline{g}\left(\vec{H}^{\Sigma}_{\overline{M}}, \gamma \right)\right\}^2+{(n-2)(n-3)C-R^{\Sigma}}<A, \nonumber \\
     \implies & \left\{\overline{g}\left(\vec{H}^{\Sigma}_{\overline{M}}, \gamma \right)\right\}^2<\frac{A+\left\{R^{\Sigma}-(n-2)(n-3)C\right\}}{(n-2)(n-3)}.
 \end{align}
 Since the scalar curvature $R^{\Sigma}$ is constant by \cite{CSY}, claim $1$ follows.\\
Claim $2$: The Weak maximum principle for the Laplacian \cite{aliasbook} holds on $\Sigma^{n-2}$.\\
Case I: $\Sigma^{n-2}$ is compact.\\
Due to \eqref{FE7} and the structural result given by \cite{CSY} the complete nontrivial Yamabe gradient soliton $(M^{n-1}, g)$ is isometric to the Riemannian product $(\mathbb{R}, dr^2)\times (\Sigma^{n-2}, g_{\Sigma})$.
Construct a smooth function $\beta: M\rightarrow \mathbb{R}$ by 
\begin{equation}\label{FE26}
    \beta(r, \theta)=r^2,
\end{equation}
where $r\in \mathbb{R}$ and $\theta\in \Sigma^{n-2}$.\\
Since $\Sigma^{n-2}$ is compact, $\beta(x)\rightarrow +\infty$ as $x=(r, \theta)\rightarrow \infty$.\\
\begin{equation*}
    \Delta \beta=2 \hspace{2mm} \forall x\in M.
\end{equation*}
Therefore, 
\begin{equation*}
    \Delta \beta\le \beta \hspace{2mm} \forall x\in M\setminus K,
\end{equation*}
where $K=[-\sqrt 2, \sqrt 2]\times \Sigma^{n-2}$, which is a compact subset of $M^{n-1}$.\\
Hence, by Theorem $2.8$ and Theorem $2.9$ of \cite{aliasbook} the weak maximum principle holds for the Laplacian on $M^{n-1}$.\\
Consider the smooth function $|\Phi|:M\rightarrow \mathbb{R}$ defined by
\begin{equation}
    |\Phi|=|\Phi^{\Sigma}_{\overline{M}}|\circ \eta,
\end{equation}
where $\eta: M\rightarrow \Sigma$ is the projection map.\\
We apply the weak maximum principle to $\Delta (|\Phi|^2)$ to get the required result.\\
Case II: $(\Sigma^{n-2}, g_{\Sigma})$ is complete and  noncompact.\\
Fix a point $p\in \Sigma^{n-2}$ and consider the distence fnction $d: \Sigma \rightarrow \mathbb{R}$ as follows:
\begin{equation*}
    d(x)=d(x,p), \hspace{2mm} \text{where $d(x, p)$ is the geodesic distance between $x$ and $p$.}
\end{equation*}
The function $d$ is smooth on $\Sigma\setminus \text{Cut}(p)$, where $\text{Cut}(p)$ is the cut locus of $p$ on $\Sigma$.
It is well known that on $\Sigma^{n-2}\setminus \text{Cut}(p)$   ,
\begin{equation}
    \Hessian (d)(X, X)(x)= \int_0^{d(x)} \left\{|\tilde{X}'(t)|^2-R^{\Sigma}\left( \tilde{X}(t), \sigma'(t), \sigma'(t), \tilde{X}(t)\right)\right\} dt,
\end{equation}
where $\tilde{X}$ is the Jacobi field along the minimal geodesic $\sigma$ satisfying $\tilde{X}(0)=0$ and $\tilde{X} (d(x))=X$. Choose a parallel orthonormal frame $\left\{e_1,..., e_{n-2}\right\}$ along $\sigma$ such that $e_1(t)= \sigma'(t)$. Consider the Jacobi fields $X_i, i=2, ..., (n-2)$  with $X_i(0)=0$ and $X_i(d(x))=e_i(d(x))$.\\
Choose $d_0 \in \left(0, d(x)\right)$, $x\in \Sigma^{n-2}\setminus \text{Cut}(p)$ such that $|Ric|\le \frac{n-3}{{d_0}^2}$ on $B_{d_0}(p)$ and $d_0< \text{inj}(p)$, where $\text{inj}(p)$ is the injectivity radius of $\Sigma^{n-2}$ at $p$ .\\ 
Constructing vector fields $Y_i$ as in \cite{JA} we get,
\begin{equation}\label{FE23}
    \Delta d(x)\le \int_0^{d_0} \left\{\frac{n-3}{{d_0}^2}- \frac{t^2}{{d_0}^2} Ric^{\Sigma}\left(\sigma'(t), \sigma'(t)\right)\right\} dt- \int_{d_0}^{d(x)} Ric^{\Sigma}\left(\sigma'(t), \sigma'(t)\right) dt,
\end{equation}
Substituing \eqref{FE1a} into \eqref{FE23} we get,
\begin{equation}
    \Delta d(x)\le \frac{n-3}{d_0}+\frac{2(n-3)}{3d_0}-\int_0^{d(x)} \left\{ (n-3)C+(n-2)g\left(\vec{H}_{\Bar{M}}^{\Sigma}, \gamma\right) B_{\Bar{M}}^{\Sigma}\left(\sigma', \sigma'\right)+ \sum_{i=1}^{n-2} B\left(\sigma', e_i\right)^2\right\} dt
\end{equation}
From \eqref{FE24} and \eqref{FE25} we come up with,
\begin{equation}
    \Delta d(x)\le \frac{5(n-3)}{3d_0}+\{2(n-2)^2 D^2+(n-3)|C|\} d(x),
\end{equation}
where $D^2= \frac{A+\left\{R^{\Sigma}-(n-2)(n-3)C\right\}}{(n-2)(n-3)}$.\\
Therefore,
\begin{equation*}
    \Delta d(x)\le E+F d(x)=G(d(x)), \hspace{2mm} \forall x\in \{\Sigma\setminus \text{cut}(p)\}\setminus K,
\end{equation*}
where $K=\left\{x\in \Sigma^{n-2}| d(x)\le d_0 \right\}$, $E=\frac{5(n-3)^2}{3d_0}>0$, $F= 2(n-2)^2D^2+(n-3)|C|>0$ and $G$ is a positive smooth function with $\frac{1}{G}\notin L^{1}(+\infty)$ and $G'(t)=F>0$.\\
Note that, $K$ is a compact subset of $\Sigma^{n-2}\setminus\text{Cut}(p)$ .\\
Hence, by applying Theorem $2.8$, Theorem $2.9$ and Remark $3.2$ of \cite{aliasbook} on $\Sigma^{n-2}\setminus \text{Cut}(p)$ we conclude that $\exists$ a sequence of points $\{p_k\}_{k}\in \Sigma^{n-2}$ such that $|\Phi_{\overline{M}}^{\Sigma}|^2 (p_k)>\sup |\Phi_{\overline{M}}^{\Sigma}|^2-\frac{1}{k}$ and $\Delta |\Phi_{\overline{M}}^{\Sigma}|^2 (p_k)< \frac{1}{k}$.\\
Hence, using claim 1 and claim 2, we conclude that on the sequence of points $\{p_k\}_{k}\in \Sigma^{n-2},$ $L(|\Phi^{\Sigma}_{\overline{M}}|^2)(p_k)<\frac{c}{k}$, where $c=2(n-2)D$.
\end{proof}

\section{\textbf{Proof of the main result}}

\subsection{Proof of Theorem \ref{thm1} }

From \eqref{FE4} we deduce that 
\begin{equation*}
    \left|\Phi^{\Sigma}_{\overline{M}}\right|^2=\left\{(n-2)(n-3)C-R^{\Sigma}\right\}+(n-2)(n-3) \left\{\overline{g}\left(\vec{H}^{\Sigma}_{\overline{M}}, \gamma \right)\right\}^2.
\end{equation*}
Throughout the proof, we denote the quantity $\overline{g}\left(\vec{H}^{\Sigma}_{\overline{M}}, \gamma \right)$ by $H^{\Sigma}_{\overline{M}}$.\\
Since the scalar curvature $R^{\Sigma}$ is constant \cite{CSY} and the operator $L$ is elliptic we get,
\begin{equation}\label{FE18}
    \frac{1}{2(n-3)} L\left(\left|\Phi^{\Sigma}_{\overline{M}}\right|^2\right)\ge \frac{1}{2} L\left\{(n-2)\left(H^{\Sigma}_{\overline{M}}\right)\right\}^2\ge H^{\Sigma}_{\overline{M}} L\left\{(n-2)H^{\Sigma}_{\overline{M}}\right\}.
\end{equation}
Now, from the definition of the linear operator $L$, we have
\begin{equation}\label{FE12}
    L\left\{(n-2)H^{\Sigma}_{\overline{M}}\right\}=\sum_{i, j} \left\{(n-2) H^{\Sigma}_{\overline{M}} \delta_{ij}-B_{ij}\right\}(n-2) \left(H^{\Sigma}_{\overline{M}}\right)_{ij}.
\end{equation}
Applying Lemma \ref{lemma2} to \eqref{FE12}, we obtain the following expression for $L\left((n-2)H^{\Sigma}_{\overline{M}}\right)$.

\begin{align}\label{FE14}
    L\left\{(n-2)H^{\Sigma}_{\overline{M}}\right\}=\sum_i (n-2)^2 H^{\Sigma}_{\overline{M}} \left(H^{\Sigma}_{\overline{M}}\right)_{ii}-\frac{1}{2} \Delta \left|B^{\Sigma}_{\overline{M}}\right|^2+ \left|\nabla B^{\Sigma}_{\overline{M}}\right|^2&+C(n-2)\left|\Phi^{\Sigma}_{\overline{M}}\right|^2-\sum_{i, j, k, m} (B_{ij} B_{km})^2 \nonumber\\
    &+(n-2)H^{\Sigma}_{\overline{M}} \sum_{i, j, m} B_{ij} B_{mi} B_{mj}.
\end{align}
On the other hand \eqref{FE4} gives,
\begin{align}\label{FE13}
    &\Delta \left|B^{\Sigma}_{\overline{M}}\right|^2= (n-2)^2\Delta \left(H^{\Sigma}_{\overline{M}}\right)^2, \nonumber\\
\implies &\frac{1}{2} \Delta \left|B^{\Sigma}_{\overline{M}}\right|^2= (n-2)^2\left\{|\nabla H^{\Sigma}_{\overline{M}}|^2+\sum_i H^{\Sigma}_{\overline{M}} \left(H^{\Sigma}_{\overline{M}}\right)_{ii}\right\}.
\end{align}
Putting \eqref{FE13} in \eqref{FE14} and using Lemma $2.1.$ of \cite{GL} we get,

\begin{equation}\label{FE15}
    L\left\{(n-2)H^{\Sigma}_{\overline{M}}\right\}\ge C(n-2)\left|\Phi^{\Sigma}_{\overline{M}}\right|^2-\sum_{i, j, k, m} (B_{ij} B_{km})^2+(n-2)H^{\Sigma}_{\overline{M}} \sum_{i, j, m} B_{ij} B_{mi} B_{mj}.
\end{equation}
Now by the construction of $\Phi^{\Sigma}_{\overline{M}}$ we get,
\begin{equation}\label{FE16}
    H^{\Sigma}_{\overline{M}} \sum_{i, j, m} B_{ij} B_{mi} B_{jm}=H^{\Sigma}_{\overline{M}} \Tr\left(\Phi^{\Sigma}_{\overline{M}}\right)^3+3 \left|\Phi^{\Sigma}_{\overline{M}}\right|^2(H^{\Sigma}_{\overline{M}})^2+(n-2)\left(H^{\Sigma}_{\overline{M}}\right)^4.
\end{equation}
 Using \eqref{FE16}, \eqref{FE4}, and the given hypothesis \ref{hyp} in \eqref{FE15} we get,
 \begin{align}\label{FE17}
     L\left\{(n-2)H^{\Sigma}_{\overline{M}}\right\}&\ge C(n-2)\left|\Phi^{\Sigma}_{\overline{M}}\right|^2-\left|\Phi^{\Sigma}_{\overline{M}}\right|^4-\frac{(n-2)(n-2-2k)}{\sqrt{(n-2)k(n-k-2)}}\left|\Phi^{\Sigma}_{\overline{M}}\right|^3H^{\Sigma}_{\overline{M}}+ (n-2) \left|\Phi^{\Sigma}_{\overline{M}}\right|^2 \left(H^{\Sigma}_{\overline{M}}\right)^2 \nonumber\\
     &= -\frac{(n-2)(n-2-2k)}{\sqrt{(n-2)k(n-k-2)}}\left|\Phi^{\Sigma}_{\overline{M}}\right|^3H^{\Sigma}_{\overline{M}}+\frac{4-n}{n-3}\left|\Phi^{\Sigma}_{\overline{M}}\right|^4+\frac{1}{n-3} \left|\Phi^{\Sigma}_{\overline{M}}\right|^2R^{\Sigma}.
 \end{align}
Adopting the notation of the paper \cite{AMP} and putting the value of $H^{\Sigma}_{\overline{M}}$ from \eqref{FE4}, we define,
\begin{align}
    Q_{R^{\Sigma}}\left(\left|\Phi^{\Sigma}_{\overline{M}}\right|\right)&= -\frac{(n-2)(n-2-2k)}{\sqrt{(n-2)k(n-k-2)}}\left|\Phi^{\Sigma}_{\overline{M}}\right|H^{\Sigma}_{\overline{M}}-\frac{n-4}{n-3}\left|\Phi^{\Sigma}_{\overline{M}}\right|^2+\frac{1}{n-3}R^{\Sigma},\nonumber\\
    Q_{R^{\Sigma}}\left(\left|\Phi^{\Sigma}_{\overline{M}}\right|\right)&= -\frac{(n-2-2k)}{\sqrt{(n-3)k(n-k-2)}}\left|\Phi^{\Sigma}_{\overline{M}}\right| \sqrt{\left|\Phi^{\Sigma}_{\overline{M}}\right|^2+\left\{R^{\Sigma}-(n-2)(n-3)C\right\}}-\frac{n-4}{n-3}\left|\Phi^{\Sigma}_{\overline{M}}\right|^2+\frac{1}{n-3}R^{\Sigma}.
\end{align}
Hence,
\begin{equation}\label{FE28}
     L\left\{(n-2)H^{\Sigma}_{\overline{M}}\right\}\ge \left|\Phi^{\Sigma}_{\overline{M}}\right|^2 Q_{R^{\Sigma}}\left(\left|\Phi^{\Sigma}_{\overline{M}}\right|\right).
\end{equation}
Combining \eqref{FE18} and \eqref{FE28} we have,
\begin{equation}\label{FE20}
    \frac{1}{2(n-3)} L\left(\left|\Phi^{\Sigma}_{\overline{M}}\right|^2\right)\ge \left|\Phi^{\Sigma}_{\overline{M}}\right|^2H^{\Sigma}_{\overline{M}}  Q_{R^{\Sigma}}\left(\left|\Phi^{\Sigma}_{\overline{M}}\right|\right).
\end{equation}
Suppose, $\sup\limits_{\Sigma}\left|\Phi^{\Sigma}_{\overline{M}}\right|^2=0$ then $\Sigma^{n-2}$ is totally umbilical.\\
Let $0<\sup\limits_{\Sigma} \left|\Phi^{\Sigma}_{\overline{M}}\right|^2< \infty$. 
Because of \eqref{FE4} we can write 
\begin{equation}
    \frac{1}{2(n-3)} L\left(\left|\Phi^{\Sigma}_{\overline{M}}\right|^2\right)\ge \left|\Phi^{\Sigma}_{\overline{M}}\right|^2 \sqrt{\left|\Phi^{\Sigma}_{\overline{M}}\right|^2+\left\{R-(n-2)(n-3)C\right\}}  Q_{R^{\Sigma}}\left(\left|\Phi^{\Sigma}_{\overline{M}}\right|\right).
\end{equation}
Therefore, by Proposition \ref{prop1} we can say that there exists a sequence of points $\{p_k\}_k\subset \Sigma^{n-2}$ such that for every $k\in \mathbb{N}$
\begin{equation}
    \frac{c}{k}>\frac{1}{2(n-3)} L\left(\left|\Phi^{\Sigma}_{\overline{M}}\right|^2\right)(p_k)\ge \left|\Phi^{\Sigma}_{\overline{M}}\right|^2(p_k) \sqrt{\left|\Phi^{\Sigma}_{\overline{M}}\right|^2(p_k)+\left\{R-(n-2)(n-3)C\right\}}  Q_{R^{\Sigma}}\left(\left|\Phi^{\Sigma}_{\overline{M}}\right|\right)(p_k).
\end{equation}
By taking $k\rightarrow \infty$ and due to $R^{\Sigma}>(n-2)(n-3)$ we get,
\begin{equation}\label{FE19}
     Q_{R^{\Sigma}}\left(\sqrt{\sup_{\Sigma}\left|\Phi^{\Sigma}_{\overline{M}}\right|^2}\right)\le 0.
\end{equation}
Using the argument of Lemma $2.7.$ and Remark $2.8$ of \cite{AMP}, we can say that if $2\le k\le (n-4)$ then the function $Q_{R^{\Sigma}}$ is a decreasing function for all $x\ge 0$ and has only one positive root, say $x^{*}$. Moreover $x^{*}$ is the expression $\sqrt{\alpha}$ written in the expression $(1.2)$ of \cite{AMP} with replacement of $n-2$ instead of $n$ and $\frac{R^{\Sigma}}{(n-2)(n-3)}$ instead of $R$. Hence, from \eqref{FE19} we can say that 
\begin{equation}
    \sup_{\Sigma} \left|\Phi^{\Sigma}_{\overline{M}}\right|^2\ge (x^{*})^2= \alpha\left(R^{\Sigma}, n-2, k, C\right).
\end{equation}
Now, Suppose $\sup\limits_{\Sigma} {\left|\Phi^{\Sigma}_{\overline{M}}\right|^2}= \alpha\left(R^{\Sigma}, n-2, k, C\right)$ and $\sup$ is attained at some point of $\Sigma^{n-2}$. Since $Q_{R^{\Sigma}}$ is a decresing function for all $x\ge 0$ then from \eqref{FE18} and \eqref{FE28} we get,
\begin{equation}
     \frac{1}{2(n-3)} L\left(\left|\Phi^{\Sigma}_{\overline{M}}\right|^2\right) \ge H^{\Sigma}_{\overline{M}}L\left((n-2)H^{\Sigma}_{\overline{M}}\right)\ge 0.
\end{equation}
Using Lemma \ref{lemma3} we conclude that $\left|\Phi^{\Sigma}_{\overline{M}}\right|^2$ and $H^{\Sigma}_{\overline{M}}$ are constant on $\Sigma^{n-2}$. Hence, from \eqref{FE4} $\left|B^{\Sigma}_{\overline{M}}\right|$ is constant on $\Sigma$. Consequently, \eqref{FE15} becomes an equality. Hence, \eqref{FE14} and \eqref{FE13} implies, $\nabla B^{\Sigma}_{\overline{M}}=0$ which implies $\Sigma^{n-2}$ is a parallel submanifold of co-dimension $2$ in $\overline{M}^n(C)$. 
\qed
\printbibliography



\end{document}